\documentclass[12pt]{article}

\usepackage{geometry}
\usepackage[T1]{fontenc}
\usepackage{lmodern}
\usepackage{microtype}
\usepackage{mathtools,amssymb,amsthm,mathrsfs}
\usepackage{enumitem}
\usepackage{xcolor}
\usepackage{booktabs}
\usepackage{hyperref}
\usepackage[nameinlink,noabbrev]{cleveref}

\hypersetup{
  colorlinks=true,
  linkcolor=blue!55!black,
  citecolor=green!40!black,
  urlcolor=blue!60!black,
  pdftitle={Todd's relation conjecture and binary relations for multiple zeta values in positive characteristic},
  pdfsubject={Multiple zeta values and Carlitz multiple polylogarithms in positive characteristic}
}

\setlist[enumerate,1]{label=\textup{(\roman*)},leftmargin=2.2em}
\setlist[itemize]{leftmargin=1.8em}
\numberwithin{equation}{section}

\newtheorem{theorem}{Theorem}[section]

\newtheorem{lemma}[theorem]{Lemma}
\newtheorem{corollary}[theorem]{Corollary}

\theoremstyle{definition}
\newtheorem{definition}[theorem]{Definition}

\newtheorem{remark}[theorem]{Remark}
\newtheorem*{acknowledgements}{\bf Acknowledgements}

\DeclareMathOperator{\Span}{Span}

\newcommand{\Fq}{\mathbb F_q}
\newcommand{\K}{K}

\newcommand{\Li}{\mathrm{Li}}
\newcommand{\Si}{\mathrm{Si}}
\newcommand{\rS}{\mathrm{S}}

\newcommand{\Fix}{\operatorname{Fix}^{\rS}}
\newcommand{\iFix}{\operatorname{Fix}^{\Si}}
\newcommand{\cB}{\mathcal B^{\rS}}
\newcommand{\cBs}{\mathcal B^{\ast \rS}}
\newcommand{\cC}{\mathcal C^{\rS}}
\newcommand{\icB}{\mathcal B^{\Si}}
\newcommand{\icBs}{\mathcal B^{\ast \Si}}
\newcommand{\icC}{\mathcal C^{\Si}}
\newcommand{\cR}{\mathscr R}
\newcommand{\cS}{\mathscr S}
\newcommand{\cZ}{\mathcal Z}
\newcommand{\cI}{\mathcal I}
\newcommand{\br}{\mathfrak{BR}^{\rS}}
\newcommand{\ibr}{\mathfrak{BR}^{\Si}}

\newcommand{\bbZ}{\mathbb Z}
\newcommand{\bbR}{\mathbb R}

\newcommand{\wt}{\operatorname{wt}}
\newcommand{\dep}{\operatorname{dep}}
\newcommand{\ts}{\operatorname{T}_{\rS}}
\newcommand{\tsi}{\operatorname{T}_{\Si}}
\newcommand{\fs}{\mathfrak{s}}
\newcommand{\fh}{\mathfrak{h}}
\newcommand{\fn}{\mathfrak{n}}
\newcommand{\fm}{\mathfrak{m}}
\newcommand{\ft}{\mathfrak{t}}
\newcommand{\fu}{\mathfrak{u}}

\newcommand{\fa}{\mathfrak{a}}
\newcommand{\fb}{\mathfrak{b}}
\newcommand{\fc}{\mathfrak{c}}
\newcommand{\fd}{\mathfrak{d}}
\newcommand{\BTTw}{\mathcal{B}_w^{(\textup{TT})}}
\newcommand{\Ss}{\ast_{\rS}}
\newcommand{\St}{\triangleleft_{\rS}}
\newcommand{\Sc}{\diamond_{\rS}}
\newcommand{\Sis}{\ast_{\Si}}
\newcommand{\Sit}{\triangleleft_{\Si}}
\newcommand{\Sic}{\diamond_{\Si}}

\title{Todd's relation conjecture and binary relations for\\
multiple zeta values in positive characteristic}
\author{Jinyuan Hu}
\date{}

\begin{document}

\maketitle

\begin{abstract}
We prove Todd's relation conjecture:
all $\Fq(\theta)$-linear relations of Thakur's multiple
zeta values are generated from
the fundamental binary relation by the operators $\cB,\cC,\cB\circ \cC$; moreover they
are also generated by $\cBs,\cC,\cBs\circ\cC$. The $\cBs$-part of this conjecture has been proved by Chang, Chen and Mishiba in \cite{CCM2023}. We prove the whole conjecture for Carlitz multiple polylogarithm values, which implies the $\cB$-part for multiple
zeta values.

We also determine all fixed relations and binary relations. Let $\br_w$ be the $\Fq(\theta)$-linear
space spanned by binary relations of weight $w$, and let $\Fix_w$ be the $\Fq(\theta)$-linear
space spanned by fixed relations. We prove that
\begin{align*}
 \sum_{w\geq1}(\dim_{\Fq(\theta)}\Fix_w)x^w
 &=\frac{x^{q+1}(1-x)}{(1-2x)(1-2x+x^{q+1})},\\
 \sum_{w\geq1}(\dim_{\Fq(\theta)}\br_w)x^w
 &=\frac{x^q(1-x)}{(1-2x)(1-2x+x^{q+1})}.
\end{align*}

Our results are based on a system of transfer theorems between multiple zeta values and multiple polylogarithm values, and a family of $\Fq(t)$-linear relation found by Im–Kim–Ngo Dac in \cite{IKN2026}.
\end{abstract}

\section{Introduction}
For any non-empty set $A$, we denote by $A^{\bullet}$ the disjoint union $\{\emptyset\} \sqcup_{k=1}^{\infty} A^k$. 
Let $\bbZ_+$ be the set of positive integers, and denote by $\cI$ the set $\bbZ_+^{\bullet}$.
The elements in $\cI$ are called \emph{tuples}.
For any tuple $\fs =(s_1,\ldots,s_r) \in \cI$, the \emph{depth} $\dep(\fs)$ and \emph{weight} $\wt(\fs)$ are defined by $\dep(\fs) = r$ and $\wt(\fs)=\sum_{i=1}^{r} s_i$, with the convention $\dep(\emptyset)=0$ and $\wt(\emptyset)=0$.
Define $\cI_w = \left\{ \fs \in \cI \mid \wt(\fs) = w \right\}$ for any weight $w \ge 0$, and define $\cI_{>0} = \cI \setminus \{\emptyset\}$.
For $w\in\bbZ_+$ and a field $F$, we denote the $F$-linear formal sum space of tuples of weight $w$ by $\cS_w^F=\bigoplus_{\fs\in\cI_w} F\fs$, and $\cS^F=\bigoplus_{\fs\in\cI} F\fs$.
As usual, $[n]$ denotes the set $\{1,2,\ldots,n\}$ for any $n \in \bbZ_+$, and $[0]=\emptyset$.

\subsection{Classical multiple zeta values}

For any tuple $\fs = (s_1,\ldots,s_r) \in \cI_{>0}$ such that $s_1>1$, the \emph{multiple zeta value} (abbreviated as MZV) $\zeta(\fs)$ is defined by the convergent nested sum
\[
\zeta(\fs) \coloneq \sum_{\substack{n_1 > n_2 > \cdots > n_r \\ n_1,n_2,\ldots,n_r \in \bbZ_+}} \frac{1}{n_1^{s_1}n_2^{s_2}\cdots n_r^{s_r}} \in \bbR.
\]
By convention, $\zeta(\emptyset) \coloneq 1$. 
These values generalize the special values $\zeta(s)$ for integers $s > 1$ of the Riemann zeta function.

There exist two products on the $\mathbb Q$-span of classical multiple zeta values, namely the shuffle and stuffle products, which give rise to numerous $\mathbb Q$-linear relations. After regularization, these lead to the extended double-shuffle relations, which are conjectured to generate all relations among classical multiple zeta values. This conjecture remains open; see \cite{Brown2012,Hoffman1992,IKZ2006,Zagier1994}.

\subsection{Thakur's function-field multiple zeta values in positive characteristic}

Let $A=\Fq[\theta]$ be the polynomial ring in one variable $\theta$ over a finite field $\Fq$ of characteristic $p$, and let $A_+$ be the set of monic polynomials in $A$. Let $K=\Fq(\theta)$ be its fraction field, and $K_\infty=\Fq((\theta^{-1}))$ be the completion of $K$ at $\infty$.

For a tuple $\fs = (s_1,\ldots,s_r) \in \cI_{>0}$ and $d\in\bbZ_{\geq0}$, define the power sum $\rS_d(\fs)$ as
\[
 \rS_d(\fs)=
 \sum_{\substack{a_1,\ldots,a_r\in A_+\\
                  d=\deg a_1>\cdots>\deg a_r}}
 \frac{1}{a_1^{s_1}\cdots a_r^{s_r}}\in K.
\]
For convenience, we set $\rS_{-1}(\emptyset)=1$, $\rS_{d}(\emptyset)=0$ for $d\in\bbZ\setminus\{-1\}$, and $\rS_d(\fs)=0$ for $d<\dep(\fs)-1$.

We further define
\[
\rS_{<d}(\fs)=\sum_{k<d}\rS_k(\fs).
\]
If we set $\fs=(s_1,\fs^-)$ (possibly $\fs^-=\emptyset$), then
\[
\rS_d(\fs)=\rS_d(s_1)\rS_{<d}(\fs^-).
\]

Now we define the multiple zeta values in positive characteristic introduced by Thakur \cite{Thakur2004} as follows:
\[
\zeta(\fs)=\sum_{d\geq0}\rS_d(\fs)\in K_{\infty}.
\]

For $\rS_d(\fs), \rS_{<d}(\fs)$ and $\zeta(\fs)$, we call $\dep(\fs)$ and $\wt(\fs)$ the depth and weight of them.

Todd and Thakur predicted an explicit basis and dimension recurrence for the
$K$-span of MZVs. The conjecture is now a theorem, independently due
to Im--Kim--Le--Ngo Dac--Pham and  Chang--Chen--Mishiba
\cite{CCM2023,IKLNDP2024}. 
Define $\cZ_{w}^{(K)}$ to be the $K$-linear subspace of $K_\infty$ spanned by MZVs $\zeta(\mathfrak{s})$ of weight $w$.

\begin{theorem}[Chang--Chen--Mishiba \cite{CCM2023}, independently Im--Kim--Le--Ngo Dac--Pham \cite{IKLNDP2024}]\label{thm_CCMIKLNP}
For any weight $w \ge 1$, we have $\dim_{K} \cZ_w^{(K)} = d_w^{(K)}$, where the sequence $\{d_w^{(K)}\}_{w \ge 1}$ is defined by the generating function 
\[
\sum_{w=1}^{\infty} d_w^{(K)} x^w = \frac{x(1-x^{q-1})}{1-2x+x^{q+1}}.
\]
Moreover, $\left\{ \zeta(\fs) ~\big|~ \fs \in \BTTw  \right\}$ is a $K$-basis of $\cZ_w^{(K)}$, where 
\[
\BTTw = \left\{ (\fh,z) \in \cI_w ~\Big|~ \fh \in [q]^{\bullet}, z \in [q-1] \right\}.
\]
\end{theorem}

\Cref{thm_CCMIKLNP} implies that there are many non-trivial $K$-linear relations of MZVs. We can obtain some of them from the binary relations.

\begin{definition}
    \begin{enumerate}
        \item A binary relation $R$ of weight $w$ is given by a collection of elements $a_i,b_j$ of $K$ such that for all $d\in\bbZ$
        \[
        \sum_ia_i\rS_d(\fs_i)+\sum_jb_j\rS_{d+1}(\ft_j)=0,
        \]
        where the sum runs through tuples $\fs_i$ and $\ft_j$ of weight $w$. We denote the above equality by $R(d)$. And we may denote $R$ by $[P;Q]$ where $P=\sum_ia_i\fs_i$ and $Q=\sum_jb_j\ft_j$.
        \item A binary relation is called a fixed relation if $b_j=0$ for all $j$.
    \end{enumerate}
\end{definition}

We denote the set of all binary relations of weight $w$ by $\br_w$ and the set of all fixed relations of weight $w$ by $\Fix_w$.

Given a binary relation
\[
R: \quad \sum_ia_i\rS_d(\fs_i)+\sum_jb_j\rS_{d+1}(\ft_j)=0,
\]
summing over all $d$:
\[
\sum_d R(d): \quad \sum_ia_i\sum_d\rS_d(\fs_i)+\sum_jb_j\sum_d\rS_{d+1}(\ft_j)=0,
\]
by definition, we obtain a $K$-linear relation of MZVs, denoted by $\overline{R}$:
\[
\sum_ia_i\zeta(\fs_i)+\sum_jb_j\zeta(\ft_j)=0.
\]

The most important binary relation, named the fundamental binary relation $R_1^{\rS}$ (see \cite{Thakur2009}, §3.4.6), is given by
\begin{equation}
    \rS_d(q)+D_1\rS_{d+1}(1,q-1)=0 \quad\text{with }D_1=\theta^q-\theta\in K.
\end{equation}

To obtain other binary relations from $R_1^S$, Todd introduced the $\cB,\cBs$ and $\cC$ operators; see \cite{Todd2018}.
\begin{definition}
    Given a tuple $W=(w_1,\cdots,w_r)\in\cI_k$.
    \begin{enumerate}
        \item The operator $\cB_W:\br_w\to\br_{w+k}$ is a linear map defined by
        \begin{equation}
            \cB_W(R): \quad\cB_W(R)(d)=\rS_d(W)\sum_{j<d}R(j) \quad \text{for any } R\in\br_w.
        \end{equation}
        \item The operator $\cBs_W:\br_w\to\br_{w+k}$ is a linear map defined by 
        \begin{equation}
        \cBs_W=\cB_{(w_1)}\circ\cdots\circ\cB_{(w_r)}   . 
        \end{equation}
        \item The operator $\cC_W:\br_w\to\br_{w+k}$ is a linear map defined by
        \begin{equation}
            \cC_W(R): \quad\cC_W(R)(d)=R(d)\rS_{<d+1}(W) \quad \text{for any } R\in\br_w.
        \end{equation}
    \end{enumerate}
    For $W=\emptyset$, we set all these operators to be the identity.
\end{definition}

To explicitly describe these operators, we recall the following product formulas due to Chang, Chen and Mishiba; see \cite{CCM2023}.
\begin{theorem}[Chang–Chen–Mishiba, \cite{CCM2023}]\label{thm:product}
    For tuples $\fs=(s_1,\fs^-),\ft=(t_1,\ft^-)$ (possibly $\fs^-=\emptyset$ or $\ft^-=\emptyset$), define the product maps $\Ss,\St$ and $\Sc$ as
    \begin{align}
        \fs\Ss\ft=&(s_1,\fs^-\Ss\ft)+(t_1,\fs\Ss\ft^-)+(s_1+t_1,\fs^-\Ss\ft^-)+\\&\sum_{0<i<s_1+t_1}\Delta_{s_1,t_1}^i(s_1+t_1-i,(i)\Ss(\fs^-\Ss\ft^-));\notag\\
        \fs\St\ft=&(s_1,\fs^-\Ss\ft);\\
        \fs\Sc\ft=&\fs\Ss\ft-(s_1,\fs^-\Ss\ft)-(t_1,\fs\Ss\ft^-).
    \end{align}
    where $\fs\Ss\emptyset=\emptyset\Ss\fs=\fs$, $\fs\Sc\emptyset=\emptyset\Sc\fs=\fs$.
    
    And $\Delta_{a,b}^i=\begin{cases}
        (-1)^{(a-1)}\dbinom{i-1}{a-1}+(-1)^{(b-1)}\dbinom{i-1}{b-1}\quad&\text{if }q-1|i,0<i<a+b;\\0&\text{otherwise.}
    \end{cases}$

    By bilinear extension, $\Ss,\St$ and $\Sc$ become maps from $\cS\times\cS$ to $\cS$. And it's easy to see that $\Ss$ and $\Sc$ are both commutative and associative. For multiple $\St$-products, we always evaluate from right to left. That is, $\fa\St\fb\St\fc$ means $\fa\St(\fb\St\fc)$.

    Then we have
    \begin{align}
        \rS_{<d}(\fs)\rS_{<d}(\ft)&=\rS_{<d}(\fs\Ss\ft);\\
        \rS_d(\fs)\rS_d(\ft)&=\rS_d(\fs\Sc\ft);\\
        \rS_d(\fs)\rS_{<d}(\ft)&=\rS_{d}(\fs\St\ft).
    \end{align}
\end{theorem}

By \Cref{thm:product}, for $W\neq\emptyset$, we may rewrite $\cB,\cBs$ and $\cC$ as 
\begin{align}
    \cB_W[P; Q]&=[W\St P + W\Sc Q+W\St Q; 0];\label{eq:defBS}\\
    \cBs_W[P; Q]&=[(w_1)\St(w_2)\St\cdots\St((w_r)\St P + (w_r)\Sc Q+(w_r)\St Q); 0];\label{eq:defBsS}\\
    \cC_W[P; Q]&=[P\Sc W+P\St W; Q\St W]\label{eq:defCS}.
\end{align}

\subsection{Main results in our paper}
As an analogue of the conjecture in the classical case, Todd formulated the following relation conjecture \cite[Conjecture 5.1]{Todd2018}, which is also the first main result in our paper.

\begin{theorem}[Todd's relation conjecture]\label{thm:Todd}
For $n\in\bbZ_{+}$, the $K$-linear relation space among Thakur MZVs of weight $q+n$ is
spanned by
\begin{equation}\label{eq:Todd-B-family-intro}
 \{\overline{\cB_U(R_1^{\rS})}: U\in\cI_n\}\cup\{\overline{\cC_V(R_1^{\rS})}: V\in\cI_n\}\cup
 \{\overline{\cB_U\circ\cC_V(R_1^{\rS})}:U\in\cI_l, V\in\cI_m, l+m=n\}.
\end{equation}
It is also spanned when every $\cB$ in \eqref{eq:Todd-B-family-intro} is replaced
by $\cBs$.
\end{theorem}

Chang, Chen, and Mishiba have proved the $\cBs$-part of \Cref{thm:Todd}. In this paper, we will prove the Carlitz-multiple-polylogarithm-value version of \Cref{thm:Todd} (see \Cref{thm:ToddCMPLV}) and a transfer theorem (see \Cref{thm:transfer}), which will yield the $\cB$-part of \Cref{thm:Todd}.

The method we use to prove the $\cB$-part of \Cref{thm:ToddCMPLV} can be applied to prove the $\cB$-part of \Cref{thm:Todd} directly. For readers concerned solely with \Cref{thm:Todd}, we refer them directly to \Cref{sec:icB}.

The second part of the paper determines all binary relations and fixed relations. Let $\Fix_w$ be the space of all fixed relations of power sums of weight $w$, and let $\br_w$ be space of all binary relations of weight $w$. The following theorem summarizes
the result.

\begin{theorem}\label{thm:main-binary}
\begin{align*}
 \sum_{w\geq1}(\dim_{K}\Fix_w)x^w
 &=\frac{x^{q+1}(1-x)}{(1-2x)(1-2x+x^{q+1})},\\
 \sum_{w\geq1}(\dim_{K}\br_w)x^w
 &=\frac{x^q(1-x)}{(1-2x)(1-2x+x^{q+1})}.
\end{align*}

Moreover, there is a non-canonical isomorphism
\[\br_w\cong\cR_w^{\zeta}\oplus\Fix_w,\]
where $\cR_w^{\zeta}$ is the $K$-vector space of all $K$-linear relations of MZVs of weight $w$.
\end{theorem}

\noindent\textbf{Statement on AI use.} All the main ideas were developed by the author.
The paper was written by the author.

ChatGPT (GPT-5.6 Sol and GPT-6 Astra, OpenAI) was used in the following ways:
\begin{itemize}
\item to assist the author in formulating \Cref{lem:BsB};
\item to compute the generating functions in \Cref{thm:main-binary};
\item to verify other proofs made by the author.
\end{itemize}

Doubao (ByteDance) was used to assist the author in writing and polishing the manuscript.

All statements, proofs, computations, and verifications contributed by the generative artificial intelligence tools mentioned above were subsequently examined, rigorously checked, and rewritten by the author, who takes full responsibility for the results. The author has no credit for this portion of the work.

\begin{acknowledgements}
    The author thanks Prof. Li Lai for his valuable suggestions and helpful comments throughout this work. The author is also deeply grateful to ITZY for the encouragement and support provided throughout the research.
\end{acknowledgements}

The paper is organized as follows. \Cref{sec:transfer} transfers our discussion of MZVs to that of Carlitz multiple polylogarithm values. \Cref{sec:proof} proves the
\Cref{thm:ToddCMPLV}. \Cref{sec:fixed} determines the fixed relations, and \Cref{sec:binary} determines the binary relations.

\section{From MZVs to CMPLVs}\label{sec:transfer}

\subsection{Carlitz multiple polylogarithm values}
We recall the notion of Carlitz multiple polylogarithm values. Set $l_0=1$ and $l_d=\prod_{k=1}^d(\theta-\theta^{q^k})$ for $d\in\bbZ_+$, for a tuple $\fs=(s_1,\cdots,s_r)\in\cI_w$, we introduce analogues of power sums
\[
\Si_d(\fs)=\sum_{d=k_1>\cdots>k_r\geq0}\frac{1}{l_{k_1}^{s_1}\cdots{l_{k_r}^{s_r}}}\in K_{\infty},
\]
Similarly, we set $\Si_{-1}(\emptyset)=1$, $\Si_{d}(\emptyset)=0$ for $d\in\bbZ\setminus\{-1\}$, and $\Si_d(\fs)=0$ for $d<\dep(\fs)-1$.

Further define
\[
\Si_{<d}(\fs)=\sum_{k<d}\Si_k(\fs)\in K_{\infty}.
\]
Similarly, we have
\[
\Si_d(\fs)=\Si_d(s_1)\Si_{<d}(\fs^-).
\]

The Carlitz multiple polylogarithm values (abbreviated as CMPLVs) are defined by
\[
\Li(\fs)=\sum_{d\geq0}\Si_d(\fs)\in K_{\infty}.
\]
For $\Si_d(\fs), \Si_{<d}(\fs)$ and $\Li(\fs)$, we still call $\dep(\fs)$ and $\wt(\fs)$ the depth and weight of them.

We can define binary relations and similar operators for $\Si$:

\begin{definition}
    \begin{enumerate}
        \item A binary relation $R$ of $\Si$ of weight $w$ is given by a collection of elements $a_i,b_j$ of $K$ such that for all $d\in\bbZ$
        \[
        \sum_ia_i\Si_d(\fs_i)+\sum_jb_j\Si_{d+1}(\ft_j)=0,
        \]
        where the sum runs through tuples $\fs_i$ and $\ft_j$ of weight $w$. We denote the above equality by $R(d)$.
        \item A binary relation is called a fixed relation if $b_j=0$ for all $j$.
    \end{enumerate}
\end{definition}

We denote the set of all binary relations of $\Si$ of weight $w$ by $\ibr_w$ and the set of all fixed relations of weight $w$ by $\iFix_w$

Similarly, we have the fundamental relation called $R_1^{\Si}$:
\begin{equation}
    \Si_d(q)+D_1\Si_{d+1}(1,q-1)=0.
\end{equation}

\begin{definition}
    Given a tuple $W=(w_1,\cdots,w_r)\in\cI_k$.
    \begin{enumerate}
        \item The operator $\icB_W:\ibr_w\to\ibr_{w+k}$ is a linear map defined by
        \begin{equation}
            \icB_W(R): \quad\icB_W(R)(d)=\Si_d(W)\sum_{j<d}R(j) \quad \text{for any } R\in\ibr_w.
        \end{equation}
        \item The operator $\icBs_W:\ibr_w\to\ibr_{w+k}$ is a linear map defined by 
        \begin{equation}
        \icBs_W=\icB_{(w_1)}\circ\cdots\circ\icB_{(w_r)}   . 
        \end{equation}
        \item The operator $\icC_W:\ibr_w\to\ibr_{w+k}$ is a linear map defined by
        \begin{equation}
            \icC_W(R): \quad\icC_W(R)(d)=R(d)\Si_{<d+1}(W) \quad \text{for any } R\in\ibr_w.
        \end{equation}
    \end{enumerate}
    For $W=\emptyset$, we set all these operators to be the identity.
\end{definition}

Similarly, we may define the product maps $\Sis,\Sit$ and $\Sic$ as
    \begin{align}
        \fs\Sis\ft=&(s_1,\fs^-\Sis\ft)+(t_1,\fs\Sis\ft^-)+(s_1+t_1,\fs^-\Sis\ft^-);\\
        \fs\Sit\ft=&(s_1,\fs^-\Sis\ft);\\
        \fs\Sic\ft=&(s_1+t_1,\fs^-\Sis\ft^-).
    \end{align}
    where $\fs\Sis\emptyset=\emptyset\Sis\fs=\fs$, $\fs\Sic\emptyset=\emptyset\Sic\fs=\fs$. And $\Sis$ and $\Sic$ are both commutative and associative. For multiple $\Sit$-products, we always evaluate from right to left.

    Then we have
    \begin{align}
        \Si_{<d}(\fs)\Si_{<d}(\ft)&=\Si_{<d}(\fs\Sis\ft);\\
        \Si_d(\fs)\Si_d(\ft)&=\Si_d(\fs\Sic\ft);\\
        \Si_d(\fs)\Si_{<d}(\ft)&=\Si_{d}(\fs\Sit\ft),
    \end{align}
which can be verified directly by definition.

And the $\icB_w,\icBs,\icC$ can be rewritten as
\begin{align}
    \icB_W[P; Q]&=[W\Sit P + W\Sic Q+W\Sit Q; 0];\label{eq:defBSi}\\
    \icBs_W[P; Q]&=[(w_1)\Sit(w_2)\Sit\cdots\Sit((w_r)\Sit P + (w_r)\Sic Q+(w_r)\Sit Q); 0];\label{eq:defBsSi}\\
    \icC_W[P; Q]&=[P\Sic W+P\Sit W; Q\Sit W]\label{eq:defCSi}.
\end{align}

\subsection{Preliminary lemmas}
CMPLVs are more tractable objects than MZVs. Meanwhile, prior work (see, e.g., \cite{Chen2015} \cite{IKN2026} \cite{Thakur2009}) enables us to transfer our discussion of MZVs to that of CMPLVs. We consider $\rS_d$ and $\Si_d$ as two linear maps from $\cS_w^K$ to $K_{\infty}$ by $K$-linear extension. We recall the following results.

\begin{theorem}[Im–Kim–Ngo Dac, \cite{IKN2026}, Theorem~3.1]\label{thm:SSiequ}
    Let $w\in\bbZ_+$. Then the $K$-vector spaces spanned by MZVs and CMPLVs of weight $w$ coincide. To be more precise, let $\fs\in\cI_w$, there exist $a_i,b_j\in\Fq$ and tuples $\fs_i,\ft_j\in\cI_w$ such that for all $d\in\bbZ$
    \begin{equation}
        \rS_d(\fs)=\sum_ia_i\Si_d(\fs_i),
    \end{equation}
    \begin{equation}
        \Si_d(\fs)=\sum_jb_j\rS_d(\ft_j).
    \end{equation}
\end{theorem}

\Cref{thm:SSiequ} ensures that we may define weight‑preserving $K$-linear maps $\ts,\tsi:\cS^K\to\cS^K$ such that $\rS_d=\Si_d\circ \ts$ and $\Si_d=\rS_d\circ \tsi$ for all $d$. But there are many choices of $\ts$ and $\tsi$, we fix
the following two canonical ones; see \cite[Proposition~3.2, Proposition~3.3]{IKN2026}. 

First, since \[\rS_d(\emptyset)=\Si_d(\emptyset),\quad\text{and}\quad \rS_d(1)=\Si_d(1);\]
we set\[\ts(\emptyset)=\tsi(\emptyset)=\emptyset,\quad\text{and}\quad \ts(1)=\tsi(1)=1.\]

Next we define $\ts$ and $\tsi$ by induction on the weight, assume that the maps have been defined for all tuples of weight less than $w$.

For $\fs=(s_1,\fs^-)$ with $\wt(\fs)=w$ and $\fs^-\neq\emptyset$, since $\ts(s_1)$ and $\ts(\fs^-)$ have been defined, we have
\[
\rS_d(\fs)=\rS_d(s_1)\rS_{<d}(\fs^-)=\Si_d(\ts(s_1))\Si_{<d}(\ts(\fs^-))=\Si_d(\ts(s_1)\Sit \ts(\fs^-)).
\]
So we may define
\begin{equation}
    \ts(\fs)=\ts(s_1,\fs^-)=\ts(s_1)\Sit \ts(\fs^-).
\end{equation}

For $\fs=(w)$, since
\begin{align*}
    \rS_d(w)&=\rS_d(w-1)\rS_d(1)-\sum_{0<i<w}\Delta_{w-1,1}^i\rS_d(w-i,i)\\&=\Si_d(\ts(w-1))\Si_d(\ts(1))-\sum_{0<i<w}\Delta_{w-1,1}^i\Si_d(\ts(w-i,i)),\
\end{align*}
we may define
\begin{equation}
    \ts(w)=\ts(w-1)\Sic \ts(1)-\sum_{0<i<w}\Delta_{w-1,1}^i\ts(w-i,i).
\end{equation}

Similarly, for $\fs=(s_1,\fs^-)$ with $\fs^-\neq\emptyset$, we define
\begin{equation}
    \tsi(\fs)=\tsi(s_1,\fs^-)=\tsi(s_1)\St \tsi(\fs^-),
\end{equation}
and
\begin{equation}
    \tsi(w)=\tsi(w-1)\Sc \tsi(1),
\end{equation}
which satisfy our desired property.

As a corollary of \Cref{thm:SSiequ}, we have 

\begin{theorem}
    Let $w\in\bbZ_+$ and $d\in\bbZ$. Then the $\Fq$-vector spaces (hence also the $K$-vector space) spanned by $\rS_d$'s and $\Si_d$'s of weight $w$ coincide.
\end{theorem}

\subsection{MZV-CMPLV transfer}
Now we can transfer our discussion of MZVs to CMPLVs. 

We first combine all $\rS_d$, obtaining a $K$-linear map called $F_w^S:\cS_w^K\to K_{\infty}^{\bbZ}$ by
\begin{equation}
    F_w^S \left( \sum a_i\fs_i \right) = \left( \sum a_i\rS_d(\fs_i) \right) _{d\in\bbZ}\in K_{\infty}^{\bbZ},
\end{equation}
similarly, we can define the map $F_w^{Si}$. 

We further define $\varphi:
K_{\infty}^{\bbZ}\oplus K_{\infty}^{\bbZ}\to K_{\infty}^{\bbZ}$ by $\varphi((A_d)_{d\in\bbZ}\oplus(B_d)_{d\in\bbZ})=(A_d+B_{d+1})_{d\in\bbZ}$, then we have the following canonical isomorphisms by definition:
\begin{align}
\br_w &\cong \ker(\varphi\circ(F_w^S\oplus F_w^S)),\quad \Fix_w\cong\ker F_w^S; \label{eq:auto1}\\
\ibr_w &\cong \ker(\varphi\circ(F_w^{Si}\oplus F_w^{Si})),\quad \iFix_w\cong\ker F_w^{Si}. \label{eq:auto2}
\end{align}

Our starting point for transferring the problem is the following theorem.

\begin{theorem}\label{thm:transfer1}
    $\ts$ is a $K$-linear automorphism of $\cS^K$ such that $F_w^S=F_w^{Si}\circ \ts$ with the inverse $\tsi$.
\end{theorem}

We need the following lemmas:
\begin{lemma}\label{lem:tst}
    For $\fa,\fb,\fc,\fd\in\cI$ with $\fa,\fc\neq\emptyset$, we have
    \begin{equation}
        (\fa\Sit\fb)\Sis(\fc\Sit\fd)=(a\Sic\fc)\Sit(\fb\Sis\fd)+\fa\Sit(\fb\Sis(\fc\Sit\fd))+\fc\Sit(\fd\Sis(\fa\Sit\fb)).
    \end{equation}
\end{lemma}
\begin{proof}
    \begin{align}
        (\fa\Sit\fb)\Sis(\fc\Sit\fd)=&(a_1,\fa^-\Sis\fb)\Sis(c_1,\fc^-\Sis\fd)\\=&(a_1+c_1,\fa^-\Sis\fc^-\Sis\fb\Sis\fd)+(a_1,(c_1,\fc^-\Sis\fd)\Sis(\fa^-\Sis\fb))\notag\\&+(c_1,(a_1,\fa^-\Sis\fb)\Sis(\fc^-\Sis\fd)),
    \end{align}
    \begin{equation}
        (a\Sic\fc)\Sit(\fb\Sis\fd)=(a_1+c_1,\fa^-\Sis\fc^-)\Sit(\fb\Sis\fd),
    \end{equation}
    \begin{equation}
        \fa\Sit(\fb\Sis(\fc\Sit\fd))=\fa\Sit(\fb\Sis(c_1,\fc^-\Sis\fd))=(a_1,\fa^-\Sis\fb\Sis(c_1,\fc^-\Sis\fd)),
    \end{equation}
    and similarly,
    \begin{equation}
        \fc\Sit(\fd\Sis(\fa\Sit\fb))=(c_1,(a_1,\fa^-\Sis\fb)\Sis(\fc^-\Sis\fd)).
    \end{equation}
    
    Combining these, we complete the proof.
\end{proof}

\begin{lemma}\label{lem:ts} For $\fa\in\cI_{>0},\fb\in\cI$, we have
    \begin{enumerate}
        \item \label{lem:i}$\ts(\fa\Ss\fb)=\ts(\fa)\Sis\ts(\fb),$
        \item $\ts(\fa\St\fb)=\ts(\fa)\Sit\ts(\fb),$
        \item \label{lem:iii}$\ts(\fa\Sc\fb)=\ts(\fa)\Sic\ts(\fb).$
    \end{enumerate}
\end{lemma}
\begin{proof}
	We prove it by induction on $\wt(\fa)+\wt(\fb)$. When $\wt(\fa)+\wt(\fb)=0,1$, it is obvious. Assume \ref{lem:i}--\ref{lem:iii} hold for $\wt(\fa)+\wt(\fb)<w$, now we first prove $\ts(\fa\Ss\fb)=\ts(\fa)\Sis\ts(\fb)$ for $\wt(\fa)+\wt(\fb)=w$. Set $\fa=(a_1,\fa^-)$ and $\fb=(b_1,\fb^-)$, possibly $\fa^-=\emptyset$, $\fb^-=\emptyset$ or $\fb=\emptyset$.
	
	First, if $\fa=(a_1,\fa^-)$ with $\fa^-\neq\emptyset$, then
	\begin{align*}
		\ts(\fa\Ss\fb)=&\ts((a_1,\fa^-\Ss\fb)+(b_1,\fa\Ss\fb^-)+(a_1+b_1,\fa^-\Ss\fb^-)\notag\\&+\sum_i\Delta_{a_1,b_1}^i(a_1+b_1-i,(i)\Ss\fa^-\Ss\fb^-))\\
		=&\ts(a_1)\Sit(\ts(\fa^-)\Sis\ts(\fb))+\ts(b_1)\Sit(\ts(\fa)\Sis\ts(\fb^-))\notag\\&+\ts(a_1+b_1)\Sit(\ts(\fa^-)\Sis\ts(\fb^-))+\notag\\&\sum_i\Delta_{a_1,b_1}^i\ts(a_1+b_1-i)\Sit(\ts(i)\Sis\ts(\fa^-)\Sis\ts(\fb^-))\\
		:=&\ts(a_1+b_1)\Sit(\ts(\fa^-)\Sis\ts(\fb^-))\notag\\&+\sum_i\Delta_{a_1,b_1}^i\ts(a_1+b_1-i)\Sit(\ts(i)\Sis\ts(\fa^-)\Sis\ts(\fb^-))+R.
	\end{align*}
    
	On the other hand, by \Cref{lem:tst}
	\begin{align*}
		\ts(\fa)\Sis\ts(\fb)=&(\ts(a_1)\Sic\ts(b_1))\Sit(\ts(\fa^-)\Sis\ts(\fb^-))+R\\=&\ts((a_1)\Sc(b_1))\Sit\ts(\fa^-\Ss\fb^-)+R
        \\=&(\ts(a_1+b_1)+\sum_i\Delta_{a_1,b_1}^i\ts(a_1+b_1-i)\Sit\ts(i))\Sit\ts(\fa^-\Ss\fb^-)+R
        \\=&\ts(a_1+b_1)\Sit(\ts(\fa^-)\Sis\ts(\fb^-))\notag\\&+\sum_i\Delta_{a_1,b_1}^i\ts(a_1+b_1-i)\Sit(\ts(i)\Sis\ts(\fa^-)\Sis\ts(\fb^-))+R.
	\end{align*}

    So $\ts(\fa\Ss\fb)=\ts(\fa)\Sis\ts(\fb)$.

    It remains to consider the case $\fa=(a)$, since $\Ss$ and $\Sis$ are both commutative, we only deal with $\fb=(b)$ and assume $a>1$ (the case $a=b=1$ is obvious).

    Since
    \begin{align*}
        \ts((a)\Ss(b))=&\ts(((a-1)\Ss(1)-(a-1,1)-(1,a-1)-\sum_i\Delta_{a-1,1}^i(a-i,i))\Ss(b))
        \\=&\ts((a-1)\Ss(b)\Ss(1))-\ts(a-1,1)\Sis\ts(b)-\ts(1,a-1)\Sis\ts(b)\notag\\&-\sum_i\Delta_{a-1,1}^i\ts(a-i,i)\Sis\ts(b),
    \end{align*}
    and
    \begin{align*}
        \ts(a)\Sis\ts(b)=&\ts((a-1)\Ss(1)-(a-1,1)-(1,a-1)-\sum_i\Delta_{a-1,1}^i(a-i,i))\Sis\ts(b)\\=&\ts((a-1)\Ss(b))\Sis\ts(1)-\ts(a-1,1)\Sis\ts(b)-\ts(1,a-1)\Sis\ts(b)\notag\\&-\sum_i\Delta_{a-1,1}^i\ts(a-i,i)\Sis\ts(b).
    \end{align*}

    It suffices to show $\ts((a-1)\Ss(b)\Ss(1))=\ts((a-1)\Ss(b))\Sis\ts(1)$.
    \begin{align*}
        \ts((a-1)\Ss(b)\Ss(1))=&\ts(((a-1,b)+(b,a-1)+(w-1)+\sum_i\Delta_{a-1,b}^i(w-1-i,i))\Ss(1))\\=&(\ts(a-1,b)+\ts(b,a-1)\\&+\sum_i\Delta_{a-1,b}^i\ts(w-1-i,i))\Sis\ts(1)+\ts((w-1)\Ss(1)),
    \end{align*}
    \begin{align*}
        \ts((a-1)\Ss(b))\Sis\ts(1)=&\ts((a-1,b)+(b,a-1)+(w-1)\\&+\sum_i\Delta_{a-1,b}^i(w-1-i,i))\Sis\ts(1).
    \end{align*}

    So the final step is to show $\ts((w-1)\Ss(1))=\ts(w-1)\Sis\ts(1)$, which follows directly from the defining recursion.

    Similarly, we can prove $\ts(\fa\St\fb)=\ts(\fa)\Sit\ts(\fb)$ and $\ts(\fa\Sc\fb)=\ts(\fa)\Sic\ts(\fb)$ for $\wt(\fa)+\wt(\fb)=w$.
\end{proof}

\begin{proof}[Proof of Theorem \ref{thm:transfer1}]
It suffices to show $\ts$ is a $K$-linear automorphism restricted to $\cS_w^K$ with inverse $\tsi$ for all $w\in\bbZ_+$. Since each $\cS_w^K$ is finite-dimensional, we only need to show $\ts\circ\tsi=\operatorname{id}$.

We still prove it by induction on the weight. When $\wt=1$, it's obvious. Assume we have done for $\wt<w$,  now we prove it for $\wt=w$

For $\fs=(s_1,\fs^-)\in\cI_w$ with $\fs^-\neq\emptyset$,
\begin{equation}
    \ts\tsi(\fs)=\ts(\tsi(s_1)\St\tsi(\fs^-))=\ts\tsi(s_1)\Sit\ts\tsi(\fs^-)=(s_1)\Sit\fs^-=\fs.
\end{equation}

For $\fs=(w)$,
\begin{equation}
    \ts\tsi(w)=\ts(\tsi(w-1)\Sc\tsi(1))=(w-1)\Sic(1)=(w).
\end{equation}
\end{proof}

\begin{corollary}\label{cor:eqrelation}
    $\tsi$ gives the isomorphism from $\iFix_w$ to $\Fix_w$, with $\tsi\oplus \tsi$ giving the isomorphism from $\ibr_w$ to $\br_w$
\end{corollary}
\begin{proof}
    Combining \eqref{eq:auto1}, \eqref{eq:auto2} and \Cref{thm:transfer1} yields the desired result.
\end{proof}

The final step is to transfer the operators $\mathcal{B}$ and $\mathcal{C}$.

\begin{theorem}\label{thm:operatortransfer}
    For $\mathcal{X}\in\{\mathcal{B}, \mathcal{C}\}$, $(\tsi\oplus\tsi)\circ\mathcal{X}_W^{Si}=\mathcal{X}_{\tsi(W)}^S\circ(\tsi\oplus\tsi)$.
\end{theorem}
\begin{proof}
    From \eqref{eq:defBS}, \eqref{eq:defCS}, \eqref{eq:defBSi}, \eqref{eq:defCSi} and \Cref{lem:ts}, the conclusion follows.
\end{proof}
\begin{remark}
    Unfortunately, \Cref{thm:operatortransfer} does not hold for $\mathcal{B}^{\ast}$. In fact, for a tuple $W=(w_1,\cdots,w_r)$, we have
    \[(\tsi\oplus\tsi)\circ\icBs_W=\cB_{\tsi(w_1)}\circ\cdots\cB_{\tsi(w_r)}\circ(\tsi\oplus\tsi),\]
    but $\tsi(w_i)$ is not a singleton tuple in general. So $\cB_{\tsi(w_1)}\circ\cdots\cB_{\tsi(w_r)}$ is not a linear combination of $\cBs$ generally.
\end{remark}

Now we can state our transfer theorem. We first state the CMPLV version of Todd's relation conjecture.

\begin{theorem}\label{thm:ToddCMPLV}
    For $n\in\bbZ_{+}$, the $K$-linear relation space among CMPLVs of weight $q+n$ is
spanned by
\begin{equation}\label{eq:ToddCMPLV-B-family-intro}
 \{\overline{\icB_U(R_1^{\Si})}: U\in\cI_n\}\cup\{\overline{\icC_V(R_1^{\Si})}: V\in\cI_n\}\cup
 \{\overline{\icB_U\circ\icC_V(R_1^{\Si})}:U\in\cI_l, V\in\cI_m, l+m=n\}.
\end{equation}
It is also spanned when every $\icB$ in \eqref{eq:ToddCMPLV-B-family-intro} is replaced
by $\icBs$.
\end{theorem}

\begin{theorem}\label{thm:transfer}
    The $\cB$-part of \Cref{thm:Todd} and the $\icB$-part of \Cref{thm:ToddCMPLV} are equivalent.
\end{theorem}
\begin{proof}
    Direct calculation shows $(\tsi\oplus\tsi)(R_1^{Si})=R_1^S$, then \Cref{cor:eqrelation} and \Cref{thm:operatortransfer} yield the desired result.
\end{proof}

\section{Proof of Theorem~2.11}\label{sec:proof}

\subsection{The $\icBs$-part}
We first prove the $\icBs$-part by using the $K$-linear relations found by Im–Kim–Ngo Dac in \cite{IKN2026}.

Given a tuple $\fs=(s_1,\cdots,s_r)\in\cI_w\setminus\BTTw$ (then $w\geq q$), we may write $\fs=(\fn,q+r,\fm)$, where $\fn=(s_1,\cdots,s_{k-1})$ with $1\leq s_1,\cdots,s_{k-1}\leq q$; either $r>0$, or $r=0$ and $\fm=\emptyset$ (in this case we set $(r,\fm)=\emptyset$). Then $\fs$ necessarily belongs to one of the following five types:
\begin{alignat*}{2}
    &\text{(Type 0)}\quad \fs=(w); &&\\
    &\text{(Type 1)}\quad \fs=(\fn,q), &&\text{with }\dep(\fn)\geq1;\\
    &\text{(Type 2)}\quad \fs=(\fn,q+r), &&\text{with }\dep(\fn)\geq1,\, r\geq1;\\
    &\text{(Type 3)}\quad \fs=(q+r,\fm), &&\text{with }r\geq1,\, \dep(\fm)\geq1;\\
    &\text{(Type 4)}\quad \fs=(\fn,q+r,\fm),\qquad &&\text{with }\dep(\fn)\geq1,\, r\geq1,\, \dep(\fm)\geq1.
\end{alignat*}

Set $\fn^+=(s_1,\cdots,s_{k-1}+1)$. For $\fs$ of Type 2 or Type 4, $\icBs_{\fn}\circ\icC_{(r,\fm)}(R^{\Si}_1)$ yields:
\begin{equation}\label{eq:typeBC}
    \Si_d(\fs)+\Si_d(\fn,q,r,\fm)+D_1\Si_d(\fn^+,(q-1)\Sis(r,\fm))+D_1\Si_d(\fn,1,(q-1)\Sis(r,\fm))=0.
\end{equation}

For $\fs$ of Type 0 or Type 3, $\icC_{(r,\fm)}(R^{\Si}_1)$ yields:
\begin{equation}\label{eq:typeC}
    \Si_d(\fs)+\Si_d(q,r,\fm)+D_1\Si_{d+1}(1,(q-1)\Sis(r,\fm))=0.
\end{equation}

For $\fs$ of Type 1, $\icBs_{\fn}(R^{\Si}_1)$ yields:
\begin{equation}\label{eq:typeB}
    \Si_d(\fs)+D_1\Si_d(\fn^+,q-1)+D_1\Si_d(\fn,1,q-1)=0.
\end{equation}

We denote the above binary relations by $\rho(\fs)=[P_{\fs};Q_{\fs}]$.

We recall the (depth, lex)-order in \cite[Section~4.3]{IKN2026}:
\begin{definition}
The (depth, lex)-order $\prec$ on $\cI$ is defined as follows.
For any two tuples $\fs, \ft \in \cI_w$, define $\fs \prec \ft$ if and only if
\begin{itemize}
\item either $\dep(\fs) < \dep(\ft)$,
\item or $\dep(\fs) = \dep(\ft)$ and $\fs$ is lexicographically smaller than $\ft$.
\end{itemize}
\end{definition}

Note that in \eqref{eq:typeBC}--\eqref{eq:typeB}, all tuples other than $\fs$ are greater than $\fs$ under the (depth, lex)-order, which implies these binary relations are $K$-linearly independent. The same conclusion is true for their corresponding relations of CMPLVs. Comparing dimensions, we complete the proof the $\icBs$-part of \Cref{thm:ToddCMPLV}

\subsection{The $\icB$-part}\label{sec:icB}
Now we prove the $\icB$-part. The key point is that we may replace $\icBs$ with $\icB$ and $\icC$.

\begin{lemma}\label{lem:BsB}
    For $\fs=(s_1,\fs^-)\in\cI$ with $\fs^-\neq\emptyset$,
    \[\icBs_{(s_1)}\circ\icB_{\fs^-}=\icB_{\fs}-\icB_{(s_1)}\circ\icC_{\fs^-}.\]
\end{lemma}
\begin{proof}
    Given a binary relation $[P;Q]\in\ibr_w$, we have
    \begin{equation}
        \icBs_{(s_1)}\circ\icB_{\fs^-}[P;Q]=[(s_1,\fs^-\Sit P+\fs^-\Sic Q+\fs^-\Sit Q);0];
    \end{equation}
    \begin{align}
        \icB_{\fs}[P;Q]=&[\fs\Sit P+\fs\Sic Q+\fs\Sit Q;0]\\=&[(s_1,\fs^-\Sis P)+\fs\Sic Q+\fs\Sit Q;0];
    \end{align}
    \begin{equation}
        \icBs_{(s_1)}\circ\icC_{\fs^-}[P;Q]=[(s_1,P\Sic\fs^-+P\Sit\fs^-)+(s_1)\Sic(Q\Sit\fs^-)+(s_1)\Sit(Q\Sit\fs^-);0];
    \end{equation}

    Since
    \[\fs^-\Sis P=\fs^-\Sit P+P\Sit\fs^-+\fs^-\Sic P.\]

    We have 
    \begin{align*}
        \icB_{\fs}[P;Q]-\icB_{(s_1)}\circ\icC_{\fs^-}[P;Q]&=\\ [(s_1,\fs^-\Sit P)+\fs\Sic Q&+\fs\Sit Q-(s_1)\Sic(Q\Sit\fs^-)-(s_1)\Sit(Q\Sit\fs^-);0].
    \end{align*}

    It suffices to show
    \[
    \fs\Sic Q+\fs\Sit Q-(s_1)\Sic(Q\Sit\fs^-)-(s_1)\Sit(Q\Sit\fs^-)=(s_1,\fs^-\Sic Q+\fs^-\Sit Q),
    \]
    We prove it for $Q=\ft=(t_1,\ft^-)$, then it holds for general $Q$.
    \begin{align*}
        \fs\Sic \ft+\fs\Sit \ft&-(s_1)\Sic(\ft\Sit\fs^-)-(s_1)\Sit(\ft\Sit\fs^-)\\&=(s_1+t_1,\fs^-\Sis\ft^-)+(s_1,\fs^-\Sis\ft)-(s_1+t_1,\ft^-\Sis\fs^-)-(s_1,t_1,\ft^-\Sis\fs^-)\\&=(s_1,\fs^-\Sic\ft+\fs^-\Sit\ft+\ft\Sit\fs^-)-(s_1,\ft\Sit\fs^-)\\&=(s_1,\fs^-\Sic \ft+\fs^-\Sit \ft).
    \end{align*}
\end{proof}

We establish the following replacement theorem:
\begin{theorem}\label{thm:BsB}
    For $W\in\cI_w$, there exist $a_{ij}\in \Fq$ and tuple pairs $(U_i,V_j)$ with $U_i\neq\emptyset$ and $\wt(U_i)+\wt(V_j)=w$ such that
    \[
    \icBs_W=\sum a_{ij}\icB_{U_i}\circ\icC_{V_j}.
    \]
\end{theorem}
\begin{proof}
    We prove this by induction on $\dep(W)$. When $\dep(W)=1$, since $\icB_{(w)}=\icBs_{(w)}$, there is nothing to prove.

    Assume the assertion holds for tuples of depth less than r, for $W=(w_1,\cdots,w_r)=(w_1,W^-)$, we have $\icBs_{W^-}=\sum a_{ij}\icB_{U_i}\circ\icC_{V_j}$, then by \Cref{lem:BsB}
    \begin{align}
    \icBs_W&=\icBs_{(w_1)}\circ\icBs_{W^-}\\&=\sum a_{ij}\icBs_{(w_1)}\circ\icB_{U_i}\circ\icC_{V_j}\\&=\sum a_{ij}(\icB_{(w_1,U_i)}-\icB_{(w_1)}\circ\icC_{U_i})\circ\icC_{V_j}\\&=\sum a_{ij}(\icB_{(w_1,U_i)}\circ\icC_{V_j}-\icB_{(w_1)}\circ\icC_{U_i\Sis V_j}),
    \end{align}
    where $\icC_U\circ\icC_V=\icC_{U\Sis V}$ can be verified by definition directly.
\end{proof}

With the $\icBs$-part already proved, the $\icB$-part of \Cref{thm:ToddCMPLV} is an immediate consequence of \Cref{thm:BsB}.

\begin{remark}
    Chang, Chen and Mishiba have proved the $\cBs$-part of \Cref{thm:Todd} (see \cite[Theorem~5.3]{CCM2023}) by directly treating MZVs and power sums. Our method in this subsection can also be used to directly complement their proof.
\end{remark}

\section{Fixed relations}\label{sec:fixed}
We determine all fixed relations in this section. Based on our discussion in \Cref{sec:transfer}, we study $\iFix_w$, which leads to the same result for $\Fix_w$.

Let $\cR_w^{\Li}=\{\sum a_i\fs_i\in\cS_w^K: \Li(\sum a_i\fs_i)=0\}$ be the set of all $K$-linear relations of CMPLVs of weight $w$, then we have natural embedding $\iota: \iFix_w\hookrightarrow\cR_w^{\Li}$, and the map $\phi:\ibr_w\to\cR_w^{\Li}$ defined by $\phi[P;Q]=P+Q$.

For $\sum a_i\fs_i\in\iFix_w$, there is no binary relation $[P;Q]$ such that $\phi[P;Q]=\iota(\sum a_i\fs_i)$ with $P,Q\notin\iFix_w$. In fact, if $[\sum b_j\ft_j;\sum c_k\fu_k]\in\ibr_w$ satisfies 
\[\phi[\sum b_j\ft_j;\sum c_k\fu_k]=\iota[\sum a_i\fs_i],\]
then we have 
\[\sum a_i\fs_i=\sum b_j\ft_j+\sum c_k\fu_k,\]
by definition, we have
\begin{equation}
    \Si_d(\sum b_j\ft_j)+\Si_d(\sum c_k\fu_k)=0,
\end{equation}
and
\begin{equation}
    \Si_d(\sum b_j\ft_j)+\Si_{d+1}(\sum c_k\fu_k)=0.
\end{equation}
Subtracting these two equations, we have
\begin{equation}\label{eq:same}
    \Si_d(\sum c_k\fu_k)=\Si_{d+1}(\sum c_k\fu_k),
\end{equation}
\eqref{eq:same} holds for all $d\in\bbZ$, it can only be that $\Si_d(\sum c_k\fu_k)=0$ for all $d\in\bbZ$. Then we have $\sum b_j\ft_j,\sum c_k\fu_k\in\iFix_w$.

Now according to our result in \Cref{sec:proof}, there exist unique $b_j\in\K$ and $\ft_j\in\cI_w\setminus\BTTw$ such that 
\[\sum a_i\fs_i=\sum b_j(P_{\ft_j}+Q_{\ft_j}),\quad\text{or equivalently}\quad\iota(\sum a_i\fs_i)=\phi(\sum b_j\rho(\ft_j)).\]

Put $\sum b_j\ft_j=\sum c_k\fc_k+\sum d_l\fd_l$, where $\fc_k=(q+r_k,\fm_k)$ is of Type 0 or Type 3, and $\fd_l$ is of other types. The above discussion yields that
\begin{align*}
    0&=\sum c_k\Si_{d+1}(1,(q-1)\Sis(r_k,\fm_k))\\&=\Si_{d+1}(1)\Si_{<d+1}(q-1)\sum c_k\Si_{<d+1}(r_k,\fm_k).
\end{align*}

So we have $\sum c_k(r_k,\fm_k)\in\iFix_{w-q}$. Define $C:\iFix_{w-q}\to\iFix_w$ by \[C(\sum c_k(r_k,\fm_k))=\sum c_k((q+r_k,\fm_k)+(q,r_k,\fm_k)+D_1(1,(q-1)\Sis(r_k,\fm_k))),\] 
then there exist unique $\sum c_k(r_k,\fm_k)\in\iFix_{w-q}$, $d_l\in K$ and $\fd_l\in\cI_w$ of Type 1, Type 2 or Type 4, such that:
\[\sum a_i\fs_i=C(\sum c_k(r_k,\fm_k))+\sum d_lP_{\fd_l}.\]

The above discussion leads to the following decomposition theorem:
\begin{theorem} For $w> q$
    \[\iFix_w= B_w\oplus C(\iFix_{w-q}),\]
    where $B_w=\Span_K\{P_{\fs}: \fs\in\cI_w\text{ is of Type 1, Type 2 or Type 4}\}$.
\end{theorem}

It remains to compute $\dim_K\iFix_w$, it's easy to see $f_w=0$ for $w\leq q$. For $w>q$, we set $f_w=\dim_K\iFix_w$ and $b_w=\dim_K B_w$, note that $C$ is injective, then
 \[f_w=b_w+f_{w-q},\]
since
\begin{align*}
b_w&=\#\{\fs\in\cI_w:\fs\text{ is of Type 1, Type 2 or Type 4}\}\\&=\#\cI_w-\#\BTTw-\#\{\fs\in\cI_w:\fs\text{ is of Type 0 or Type 3}\}\\&=2^{w-1}-d^{(K)}_w-2^{w-q-1},
\end{align*}
we have 
\begin{equation}\label{eq:feq}
    f_w=2^{w-1}-d^{(K)}_w-2^{w-q-1}+f_{w-q}.
\end{equation}

Set $w=mq+r$ with $0< r \leq q$, by iteration, we have
\begin{equation}
    f_w=2^{w-1}-2^{r-1}-\sum_{k=0}^{m-1}d^{(K)}_{w-kq},
\end{equation}
which gives a formula to calculate $f_w$.

We may also calculate the generating function of $f_w$. Set $r_w=2^{w-1}-d_w^{(K)}$, by \Cref{thm_CCMIKLNP}, we have 
\[\sum_{w=1}^{\infty}r_wx^w=\frac{x^q(1-x)^2}{(1-2x)(1-2x+x^{q+1})}.\]

Multiply both sides of \eqref{eq:feq} by $x^w$, and sum over all $w> q$, we obtain:
\[\sum_{w=1}^{\infty}f_wx^w=\sum_{w=1}^{\infty}r_wx^w-x^q-\frac{x^{q+1}}{1-2x}+x^q\sum_{w=1}^{\infty}f_wx^w,\]
so
\begin{equation}
    \sum_{w=1}^{\infty}f_wx^w=\frac{x^{q+1}(1-x)}{(1-2x)(1-2x+x^{q+1})}.
\end{equation}

\section{All binary relations }\label{sec:binary}
As the final section of the paper, we determine $\ibr_w$.

Recall the map $\phi:\ibr_w\to\cR_w^{\Li}$ defined by $\phi[P;Q]=P+Q$ in \Cref{sec:fixed}, and \Cref{thm:ToddCMPLV} ensures that $\phi$ is surjective. So we just need to study $\ker{\phi}$.

\begin{theorem}\label{thm:ker}
    \[\ker{\phi}\cong \iFix_w.\]
    
    To be more precise, a binary relation $[P;Q]\in\ker{\phi}$ if and only if $P\in\iFix_w$ and $P=-Q$.
\end{theorem}
\begin{proof}
    By definition, $[P;Q]\in\ker{\phi}$ if and only if $P=-Q$. Obviously $[P;-P]\in\ker{\phi}$ for $P\in\iFix_w$. We only need to show that if $[P;-P]\in\ibr_w$, then $P\in\iFix_w$.

    In fact, since $[P;-P]\in\ibr_w$, we have $\Si_d(P)=\Si_{d+1}(P)$ for all $d\in\bbZ$. $\Si_d(P)=0$ for $d<0$, then it holds for all $d\in\bbZ$ by induction.
\end{proof}

Combining \Cref{thm:ker} and previous results, we have
\begin{theorem}
    \[\ibr_w\cong\cR_w^{\Li}\oplus\iFix_w,\]
    and
    \[\sum_{w=1}^{\infty}(\dim_K\ibr_w)x^w=\frac{x^q(1-x)}{(1-2x)(1-2x+x^{q+1})}.\]
\end{theorem}

The following corollary was suggested by Li Lai, which provides an alternative perspective on the relationship between $\ibr$ and $\iFix$:
\begin{corollary}
    $\icB_{(1)}$ gives an isomorphism from $\ibr_{w}$ to $\iFix_{w+1}$.
\end{corollary}
\begin{proof}
    By \eqref{eq:defBSi}, the image of $\ibr_w$ under $\icB_{(1)}$ lies in $\iFix_{w+1}$, so this map is well‑defined. Generating functions show that these two spaces have the same dimension. It remains to show that $\icB_{(1)}$ is injective.

    For $\fc=(c_1,\fc^-)\in\cI_{w}$, set $\prescript{+}{}{\fc}=(c_1+1,\fc^-)$. For $[\sum b_j\fb_j;\sum c_k\fc_k]\in\ibr_{w}$, we have
    \begin{align*}
        \icB_{(1)}\Bigl[\sum b_j\fb_j;\sum c_k\fc_k\Bigr]
        &=\biggl[\sum b_j(1,\fb_j)+\sum c_k\prescript{+}{}{\fc_k}+\sum c_k(1,\fc_k);\,0\biggr]\\
        &=\biggl[\Bigl(1,\sum b_j\fb_j+\sum c_k\fc_k\Bigr)+\sum c_k\prescript{+}{}{\fc_k};\,0\biggr].
    \end{align*}
    Suppose that
    \[
    \Bigl(1,\sum b_j\fb_j+\sum c_k\fc_k\Bigr)+\sum c_k\prescript{+}{}{\fc_k}=0.
    \]
    Since the first component of each $\prescript{+}{}{\fc_k}$ is greater than $1$, we must have $\sum c_k\prescript{+}{}{\fc_k}=0$, which in turn implies $\sum c_k\fc_k=0$ and $\sum b_j\fb_j=0$.
\end{proof}

\vspace*{3mm}
\begin{flushright}
\begin{minipage}{148mm}\sc\footnotesize
J.\,Hu: School of Mathematical Sciences, Fudan University, Shanghai, China \\
{\it E-mail address}: \href{mailto:jinyuanhu2004@gmail.com}{{\tt jinyuanhu2004@gmail.com}}
\vspace*{3mm}
\end{minipage}
\end{flushright}

\end{document}